\documentclass[centering,11pt]{amsart}

\usepackage[a4paper]{geometry}

\usepackage{todonotes}

\usepackage[colorlinks,cite color=blue,pagebackref=true,pdftex]{hyperref}

\usepackage{amsfonts,amsmath,amssymb,amsthm}
\usepackage{url}

\usetikzlibrary{shapes,snakes}

\theoremstyle{plain}
\newtheorem{theorem}{Theorem}[section]

\newtheorem{lemma}[theorem]{Lemma}
\newtheorem{proposition}[theorem]{Proposition}
\newtheorem{corollary}[theorem]{Corollary}

\theoremstyle{definition}
\newtheorem{remark}[theorem]{Remark}
\newtheorem{example}[theorem]{Example}
\newtheorem{question}{Question}

\renewcommand\emptyset{\varnothing}
\newcommand{\RR}{\mathbb{R}}
\newcommand{\ZZ}{\mathbb{Z}}
\newcommand{\conv}{\operatorname{conv}}
\newcommand{\cone}{\operatorname{cone}}
\newcommand\Def[1]{\textbf{#1}}
\newcommand\Ind{\mathcal{I}}
\newcommand\Bases{\mathcal{B}}
\newcommand\Cocircuits{\mathcal{C}^*}

\newcommand\Zero{\boldsymbol{0}}
\newcommand{\TTmax}{\mathbb{T}_{\max}}
\newcommand\tconv{\operatorname{tconv}}
\newcommand\SetOf[2]{\left\{\left.#1\vphantom{#2}\ \right|\ #2\vphantom{#1}\right\}}
\newcommand\neighbour{\mathcal{N}}

\newcommand{\nbhdG}{\mathcal{N}_G}

\DeclareMathOperator{\supp}{supp}

\title{Tropical Carath\'eodory with Matroids}

\author{Georg Loho}
\address{Department of Mathematics\\
  London School of Economics and Political Science\\
  Houghton Street\\
  London\\
  WC2A 2AE\\
  UK}
\email{g.loho@lse.ac.uk}

\author{Raman Sanyal}
\address{FB 12 - Institut für Mathematik\\
  Goethe-Universität Frankfurt\\
  Robert-Mayer-Str. 10\\
  D-60325 Frankfurt am Main\\
  Deutschland}
\email{sanyal@math.uni-frankfurt.de}

\date{\today}

\begin{document}
\begin{abstract}
    B\'ar\'any's colorful generalization of Carath\'eodory's Theorem combines
    geometrical and combinatorial constraints. Kalai--Meshulam (2005) and
    Holmsen (2016) generalized B\'ar\'any's theorem by replacing color classes
    with matroid constraints. In this note, we obtain corresponding results in
    tropical convexity, generalizing the tropical colorful Carath\'eodory
    Theorem of Gaubert--Meunier (2010). Our proof is inspired by geometric
    arguments and is reminiscent of matroid intersection. In particular, we
    show that the topological approach fails in this setting.

    We also discuss tropical colorful linear programming and show that it is
    NP-complete. We end with thoughts and questions on generalizations to
    polymatroids, anti-matroids as well as examples and matroid simplicial
    depth.
\end{abstract}

\maketitle

\section{Introduction}\label{sec:intro}

Imre B\'ar\'any's colorful version of the Carath\'eodory
theorem is a gem of discrete geometry.

\begin{theorem}[{\cite[Thm.~2.1]{Barany:1982}}]\label{thm:Barany}
    Let $C_1,\dots,C_{d+1} \subset \RR^d$ be finite sets of points such that
    $0$ is contained in the convex hull $\conv(C_i)$ for all $i=1,\dots,d+1$.
    Then there are $p_i \in C_i$ for $i=1,\dots,d+1$ such that $0 \in
    \conv(p_1,\dots,p_{d+1})$.
\end{theorem}

The sets $C_i$ are called \Def{color classes} and $\{p_1,\dots,p_{d+1}\}$ is
called a \Def{colorful simplex}. If $C_1 = \cdots = C_{d+1}$ then this
recovers Carath\'eodory's original result but B\'ar\'any's theorem exhibits a
much more intricate and beautiful interplay of geometry and combinatorics.

\newcommand\rk{\rho}%
There have been many generalizations of Theorem~\ref{thm:Barany} mostly
weakening the prerequisites for the existence of a colorful simplex containing
the origin; see Section~\ref{sub:holmsen} as well
as~\cite{HolmsenPachTverberg:2008, ArochaBaranyBrachoFabilaMontejano:2009,
BokowskiBrachoStrausz:2011}.  Kalai and Meshulam~\cite{KalaiMeshulam:2005}
gave a different generalization of the Colorful Carath\'eodory Theorem by
interpreting the colorful condition in terms of matroids. A \Def{matroid} $M =
(E,\Ind)$ consists of a finite ground set $E$ and a non-empty collection of
subsets $\Ind \subseteq 2^E$ satisfying the following conditions: $\Ind$ is
closed under taking subsets and for any $I,J \in \Ind$ with $|I| < |J|$ there
is $e \in J \setminus I$ such that $I\cup e \in \Ind$. The \Def{rank function}
associated to $M$ is defined as
\[
    \rk(A) \ := \ \max( |I| : I \subseteq A, I \in \Ind ) \, .
\]
We refer the reader to~\cite{Oxley:2011} for more on matroids.

\begin{theorem}[{\cite[Corollary~1.4]{KalaiMeshulam:2005}}] \label{thm:KM}
    Let $M$ be a matroid on ground set $E$ and $V : E \to \RR^d$. Assume that
    $0 \in \conv(V(S))$ for every $S \subseteq E$ with $\rho(S) = \rho(E)$ and
    $\rho(E\setminus S) \le d$. Then there exists an independent set $T \in
    \Ind$ such that $0 \in \conv(V(T))$.
\end{theorem}

If $E = E_1 \sqcup E_2 \sqcup \cdots \sqcup E_{d+1}$ and $I \in \Ind$ if and
only if $|I \cap E_i| \le 1$ for all $1 \le i \le d+1$, then $M$ is called a
\Def{partition matroid} and Theorem~\ref{thm:KM} implies
Theorem~\ref{thm:Barany} with color classes $C_i = V(E_i)$.  In fact,
Theorem~\ref{thm:KM} yields a slightly stronger version; see discussion
preceding Corollary~\ref{cor:Barany_stronger}. The rank conditions are a bit
tricky to interpret (but quite natural to the proof
in~\cite{KalaiMeshulam:2005}) and we offer a simple combinatorial
reformulation in Section~\ref{sec:matroids} and Theorem~\ref{thm:tropicalKM}
below.

Other geometric setups in which the (Colorful) Carath\'eodory theorem holds
have been considered, for example tropical geometry. The \Def{(max-)tropical
semiring} (or \Def{max-plus-semiring}) is the set $\TTmax = \RR \cup
\{-\infty\}$ together with tropical addition and multiplication given by $a
\oplus b := \max(a,b)$ and $a \odot b := a+b$. \emph{Tropical mathematics} is
a comparatively young but vibrant area of research that offers new exciting
perspectives in virtually all areas of mathematics including algebra,
geometry, game theory, and optimization. For tropical convex geometry, one
defines the \Def{tropical convex hull} of a finite set $V = \{ v^{1}, \dots,
v^{n} \} \subset \TTmax^d$ by
\begin{equation} \label{eq:convex+hull}
    \tconv{V} \ := \ \SetOf{\bigoplus_{j=1}^{n} \lambda_j \odot v^{j}}{\lambda_j
    \in \TTmax \ , \ \bigoplus_j \lambda_j = 0} \enspace .
\end{equation}

We refer the reader to~\cite{Butkovic:2010,
AllamigeonBenchimolGaubertJoswig:2015} for a thorough introduction to tropical
convexity and its connection with classical convexity. Tropical convexity
turns out to be useful in gaining new insights into convex optimization;
cf.~\cite{AllamigeonBenchimolGaubertJoswig:2018}. The combinatorics of
tropical convex hulls is further studied in~\cite{Joswig:2020}. 
A first goal of this paper is the study of matroid generalizations of the
Colorful Carath\'eodory Theorem in tropical convex geometry. In
Section~\ref{sub:greedy}, we prove the following tropical version of
Theorem~\ref{thm:KM}.

\begin{theorem} \label{thm:tropicalKM}
    Let $M = (E,\Ind)$ be a matroid and $V : E \to \TTmax^d$. Assume that
    \begin{equation} \label{eq:tropical+BC+condition}
        0 \in \tconv(V(B \cup C)) \quad \text{ for every basis } B \text{ and
        cocircuit } C \text{ of } M.
    \end{equation}
    Then there exists a basis $B_0$ such that $0 \in \tconv(V(B_0))$.
\end{theorem}

Gaubert and Meunier \cite{GaubertMeunier:2010} proved a tropical Colorful
Carath\'eodory Theorem that, as with Theorem~\ref{thm:KM}, is implied by our
result.  Whereas the tropical Colorful Carath\'eodory Theorem is a consequence
of the pigeonhole principle, our proof is inspired by matroid intersection and
the geometric proof of Theorem~\ref{thm:KM} that we give in
Section~\ref{sub:cocircuits}.

Theorem~\ref{thm:Barany} was generalized by Holmsen~\cite{Holmsen:2016} to
matroids and oriented matroids. In Section~\ref{sub:holmsen}, we discuss the
relation to Theorem~\ref{thm:KM} and we prove a slightly weaker analogue for
tropical convexity. 

The collection of subsets of $E$ \emph{not} containing $0$ in the convex hull
is an abstract simplicial complex, called the \emph{support complex}. The
proofs of Kalai--Meshulam and Holmsen build on homological properties of the
support complex. In Section~\ref{sub:homological} we explain that while the
results remain true for tropical convexity, the homological methods fail
badly.

In Section~\ref{sec:LP}, we study tropical colorful linear programming and we
show that just like its classical counterpart, it is NP-complete. We also
remark that it is generally not true that tropicalizations of hard problems
are hard by considering tropical $0/1$-integer linear programming.

The paper closes with Section~\ref{sec:end} with afterthoughts and questions
on generalizations to polymatroids, anti-matroids, examples, and matroid
simplicial depth.

\textbf{Acknowledgements.} This work started during the program \emph{Tropical
Geometry, Amoebas and Polytopes} at the Institute Mittag-Leffler in spring
2018. The authors would like to thank the institute as well as the organizers
for the inspiring atmosphere.  We would also like to thank Andreas Holmsen,
Diane Maclagan, Sebastian Manecke, Fr{\'e}d{\'e}ric \mbox{Meunier}, Matthias
Schymura, and Bilal Sheikh for insightful conversations.  GL was supported by
the European Research Council (ERC) Starting Grant ScaleOpt, No.~757481. 

\section{Matroids and tropical convexity}\label{sec:matroids}

\subsection{Bases and Cocircuits}\label{sub:cocircuits}

We start by giving a simpler combinatorial perspective on
Theorem~\ref{thm:KM}.  The \Def{bases} of a matroid $M = (E,\Ind)$ is the
collection $\Bases \subset \Ind$ of inclusion-maximal independent sets. The
basis exchange property states that for any $B_1, B_2 \in \Bases$ and $e \in
B_1 \setminus B_2$ there is $f \in B_2 \setminus B_1$ such that $(B_1
\setminus e) \cup f \in \Bases$.  The \Def{circuits} of $M$ are the
inclusion-minimal dependent subsets and the \Def{dual} of $M$ is the matroid
$M^*$ with bases $\Bases^* = \{E \setminus B : B \in \Bases\}$. The
\Def{cocircuits} $\Cocircuits$ of $M$ are the circuits of $M^*$. It is easy to
see that each cocircuit of $M$ is a transversal of $\Bases$, that is, any
$C \in \Cocircuits$ has a non-empty intersection with every basis $B \in
\Bases$. In fact, for any element $e$ in a given basis $B$ there is $C \in
\Cocircuits$ with $C \cap B = \{ e \}$. We call $C$ the \Def{fundamental
cocircuit} for $B$ and $e$.

Note that Theorem~\ref{thm:KM} is trivial for $\rk(E) < d+1$. Thus, by truncation,
it suffices to assume that $\rk(E) = d+1$. The condition $\rk(E \setminus S)
\le d$ then states that $S$ meets every basis, while $\rk(S) = \rk(E)$
means that $S$ contains a basis. This gives the following reformulation of
Theorem~\ref{thm:KM}.

\begin{theorem}[Reformulation of Thm.~\ref{thm:KM}]
\label{thm:reformulated+KM}
    Let $M$ be a matroid on the ground set $E$ and $V : E \to
    \RR^d$. Assume that
    \[
        0 \in \conv(V(B \cup C)) \quad \text{ for all } B \in \Bases \text{ and
        } C \in \Cocircuits \, .
    \]
    Then there exists a basis $B_0$ such that $0 \in \conv(V(B_0))$.
\end{theorem}

If $M$ is a partition matroid with $E = E_1 \sqcup \cdots \sqcup E_{d+1}$,
then the bases are given by sets $B$ with $|B \cap E_i| = 1$. The cocircuits
are exactly $\Cocircuits = \{E_1,\dots,E_{d+1}\}$. This shows that
Theorem~\ref{thm:KM} yields a strengthening of Theorem~\ref{thm:Barany}: If for
all $i$ and any choice $p_j \in C_j$ with $j\neq i$ the origin in contained in
the convex hull of $C_i \cup \{p_j : j \neq i\}$, then $0$ is contained in
the convex hull of a colorful simplex; see Figure~\ref{fig:partition-KM}.
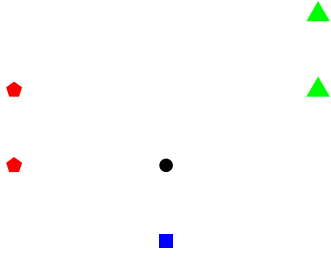
\begin{figure}[ht]
    \begin{tikzpicture}
        \tikzset{RedOuter/.style = {regular polygon,regular polygon sides=5,
                inner sep=1.8,fill=red}}
        \tikzset{BlueOuter/.style = {regular polygon,regular polygon sides=4,
                inner sep=1.8,fill=blue}}
        \tikzset{GreenOuter/.style = {regular polygon,regular polygon sides=3,
                inner sep=1.8,fill=green}}
        \tikzset{Origin/.style = {circle, fill=black, inner sep=1.8}}
  
        \coordinate (origin) at (0,0){};

        \coordinate (A1) at (-2,1);
        \coordinate (A2) at (-2,0);

        \coordinate (B1) at (2,1);
        \coordinate (B2) at (2,2);

        \coordinate (C1) at (0,-1);

        \node[Origin] at (origin){};
        \node[RedOuter] at (A1){};
        \node[RedOuter] at (A2){};
        \node[GreenOuter] at (B1){};
        \node[GreenOuter] at (B2){};
        \node[BlueOuter] at (C1){};
    \end{tikzpicture}
    \caption{The black circle is contained in the convex hull of any color
    class with the additional choice of an point from each of the other color
    classes.}\label{fig:partition-KM}
\end{figure}

We give a proof of Theorem~\ref{thm:reformulated+KM} extending the original
arguments of B\'ar\'any from~\cite{Barany:1982}; see also the last remark on
page $4$ of~\cite{Holmsen:2016}. We need the following Lemma.

\begin{lemma}[{\cite[Lem.~2]{Dawson:1980}}] \label{lem:augment+transversal}
    Let $M = (E,\Ind)$ be a matroid and $A \in \Ind$. For $e \in E \setminus
    A$, we have $A \cup e \in \Ind$ if and only if there is a cocircuit $C \in
    \Cocircuits$ with $e \in C$ and $C \cap A = \emptyset$.
\end{lemma}

\begin{proof}[Proof of Theorem~\ref{thm:KM}]
    Let $B$ be a basis of $M$ and let $\delta$ be the distance of the convex
    hull $\conv(V(B))$ from the origin.  If $\delta$ is zero, we are done, so
    assume the contrary.  Let $z$ be the point of minimal distance from the
    origin in $\conv(V(B))$ and $H_z$ be the affine hyperplane with normal
    vector $z$ through $z$.  As $z$ is contained in the supporting hyperplane
    $H_z$ of $\conv(V(B))$, there is a subset $I$ of $B$ such that $z \in
    \conv(V(I))$.  Let $C$ be a fundamental cocircuit for an element in $B
    \setminus I$.  As the origin is contained in $\conv(V(B \cup C))$, there
    is an element $c \in C \setminus B$ such that $c$ is in the same open
    halfspace of $H_z$ as the origin.  By Lemma~\ref{lem:augment+transversal},
    $B' = B \setminus C \cup \{c\}$ is a basis. The distance of $\conv(V(B'))$
    to the origin is strictly smaller than $\delta$.  As there are only
    finitely many bases, the claim follows. 
\end{proof}

B\'ar\'any's Theorem~2.3 in~\cite{Barany:1982} slightly strengthens the
Colorful Carath\'eodory theorem~\ref{thm:Barany} in that any $p_1 \in C_1$ can
be completed to a colorful simplex containing the origin. The same proof as
above yields the generalization to the matroid setting.

\begin{corollary}\label{cor:Barany_stronger}
    Let $M = (E,\Ind)$ be a matroid and $V : E \to \RR^d$ satisfying the
    condition of Theorem~\ref{thm:KM}. If $e \in E$ is not a loop, then there
    is a basis $B_0 \in \Bases(M)$ with $e \in B_0$ and $0 \in \conv(V(B_0))$.
\end{corollary}
\begin{proof}
    Instead of convex hulls consider the finitely generated convex cones
    $\cone V(A)$ for $A \subseteq E$. The prerequisites of
    Theorem~\ref{thm:reformulated+KM} then dictate that $\cone V(B \cup C) =
    \RR^d$. It suffices to show that $-V(e)$ is contained in some basis of the
    contracted matroid $M/e$ with bases $\Bases' = \{ B \setminus e : B \in
    \Bases, e \in B \}$. If this is not the case, then consider a linear
    hyperplane supporting $\bigcup_{B \in \Bases'} \cone(V(B))$ and separating
    it from $-V(e)$. Now a similiar argument to that in the proof of
    Theorem~\ref{thm:reformulated+KM} yield the claim.
\end{proof}

In terms of general set-systems, the collection of bases of a matroid is a
\Def{clutter} (or \Def{Sperner family}):  $B \not\subset B'$ for all $B,B' \in
\Bases$. The cocircuits form a \Def{blocker} for $\Bases$: $C \cap B \neq
\emptyset$ for all $C \in \Cocircuits$ and $B \in \Bases$ and the elements $C
\in \Cocircuits$ are inclusion-minimal with this property. 
Minimality shows that $\Cocircuits$ is a clutter with blocker $\Bases$.
This property of a dual pair of clutters $\mathcal{C}$ and $\mathcal{D}$,
see~\cite[Corollary]{EdmondsFulkerson:1970}, ensures the generalization of the
concept of a fundamental cocircuit: for every $H \in \mathcal{C}$ and $h \in
H$ there is $K \in \mathcal{D}$ such that $H \cap K = \{h\}$.
The following
proposition states that the proof of Theorem~\ref{thm:KM} is \emph{optimal} in
the sense that it does not apply to more general clutters.

\begin{proposition} \label{prop:blocker+characterization+matroid}
    A clutter $\mathcal{C}$ is the set of bases of a matroid if and only if
    $(H \setminus K) \cup k$ is in $\mathcal{C}$ for every $k \in K$.
\end{proposition}
\begin{proof}
  Sufficiency follows from the definition of
    cocircuits and Lemma~\ref{lem:augment+transversal}. 
    For necessity one easily verifies the basis exchange axiom: Let $H_1, H_2
    \in \mathcal{C}$. Fix an element $h \in H_1 \setminus H_2$. Let $K \in
    \mathcal{D}$ with $H_1 \cap K = \{h\}$. Since $K$ intersects every element
    of $\mathcal{C}$, there is an element $k \in K \cap H_2$. By assumption,
    $H_1 \setminus \{h\} \cup \{k\}$ is again an element of $\mathcal{C}$.
    This shows the claim. 
\end{proof}

\begin{remark}
    \newcommand\Puiseux{\RR\{\!\!\{t\}\!\!\}}
    Note that the proofs in this section apply to general ordered fields, such
    as the field of Puiseux series $\Puiseux$.  In particular, for tropical
    configurations $V : E \to \TTmax^d$ that arise as tropicalizations of some
    $\hat V : E \to \Puiseux^d$, Theorem~\ref{thm:KM} implies
    Theorem~\ref{thm:tropicalKM} as well as the results in the next section.
    However, the containment structure of a tropical convex hull is not
    necessarily preserved under lifting to $\Puiseux^d$.  This can be seen
    from the proof of~\cite[Proposition~2.1]{DevelinYu:2007} that each
    tropical polytope occurs as the image of a polytope over Puiseux series.
    The lift of a point, which is contained in the tropical convex hull of a
    finite set $V(E)$ of points, is not necessarily contained in the convex
    hull of the lifted points. 
\end{remark}

\subsection{Tropical Greedy Bases} \label{sub:greedy}

By virtue of tropical convexity, the proof of Theorem~\ref{thm:tropicalKM}.
reduces to a simpler claim about bipartite graphs.  To a point $p \in
\TTmax^d$, we associate a decomposition of $\TTmax^d$ into $d+1$ affine
sectors. For $i \in [d+1] := \{1,\dots,d+1\}$, we define the \Def{$i$-th
affine sector} as
\begin{equation} \label{eq:aff+point+sector}
    S_i(p) \ := \ \SetOf{z \in \TTmax^d}{z_i + p_k \geq z_k + p_i \text{ for }
    1 \le k \le d+1} ,
\end{equation}
where we set $p_{d+1} := z_{d+1} := 0$.  For a point $v \in \TTmax^d$ we set
\[
    \neighbour_p(v) \ := \ \SetOf{i \in [d+1]}{v \in S_i(p)}.
\]
The \Def{covector graph} $G_V$ of $p$ with respect to $V$ is the bipartite
graph on $V \sqcup [d+1]$ with edges $(v,i)$ for $i \in \neighbour_p(v)$.  If
the covector graph $G$ is fixed, we will use $\neighbour_G(v) \subseteq [d+1]$
for neighborhood of $v \in V$.

The next Lemma characterizes the containment of a point with finite coordinates; we omit the version for points with $-\infty$ entries as we do not need it in the following.  

\begin{lemma}[{\cite[Prop.~9]{DevelinSturmfels:2004},
    \cite[Lemma~28]{JoswigLoho:2016}}] \label{lem:tropical+Farkas}
    The point $p \in \RR^d$ is contained in the tropical convex hull of $v^1,
    \ldots, v^n$ if and only if no node of $[d+1]$ is isolated in $G_V$.
\end{lemma}

Note that, by construction, no node of $V$ is isolated in $G_V$.

\begin{proof}[Proof of Theorem~\ref{thm:tropicalKM}]
    Let $G = G_V$ be the covector graph for $p = 0$ and $V = V(E)$.  In
    analogy to our proof of Theorem~\ref{thm:reformulated+KM}, we start with
    an arbitrary basis $B$ of $M$. If $0$ is not contained in the tropical
    convex hull of $V(B)$ there is an element $\ell \in [d+1]$ not covered by
    the neighborhood $\nbhdG(B) = \bigcup_{e \in B} \nbhdG(e)$ in $G$.  Using
    the tropical colorful Carath\'eodory theorem (i.e., the pigeonhole principle
    applied to the covector graph $G$), there is an element $b \in B$ with
    $\nbhdG(B) = \nbhdG(B \setminus b)$.
    By~\eqref{eq:tropical+BC+condition}, the neighborhood $\nbhdG(B \cup C)$
    for each basis $B$ and cocircuit $C$ of $M$ is $[d+1]$.  In particular,
    this holds for the fundamental cocircuit $C$ of $B$ and $b$.  Hence, let
    $c \in C$ be an element adjacent to $\ell$.  Setting $B' = B \setminus C
    \cup \{c\}$, we have $|\nbhdG(B')| \geq |\nbhdG(B)| + 1$.  Iterating this
    construction at most $d$ times yields a desired basis. 
\end{proof}

The proof gives rise to an algorithm which uses an oracle for checking
independence or providing a fundamental cocircuit. 

Note that in the course of the proof, we can ensure that a fixed element $e
\in B$ does not leave the basis. This gives the tropical version of
Corollary~\ref{cor:Barany_stronger}.
\begin{corollary}
    Let $M$ and $V : E \to \TTmax^d$ satisfy the conditions of
    Theorem~\ref{thm:tropicalKM}. If $e \in E$ is not a loop, then there is a
    basis $B_0 \in \Bases$ such that $0 \in \tconv(V(B_0))$.
\end{corollary}

\subsection{Unions of color classes} \label{sub:holmsen}

Another strengthening of the Colorful Carath\'eodory Theorem is obtained by
considering unions of color classes. 

\begin{theorem}[{\cite[Theorem~1]{ArochaBaranyBrachoFabilaMontejano:2009},\cite[Theorem~5]{HolmsenPachTverberg:2008}}]  \label{thm:strong+Barany}
    Let $C_1,\dots,C_{d+1} \subset \RR^d$ be finite sets of points such that
    $0$ is contained in the convex hull $\conv(C_i \cup C_j)$ for all
    $i \neq j$.  Then there are $p_i \in C_i$ for $i=1,\dots,d+1$ such
    that $0 \in \conv(p_1,\dots,p_{d+1})$.
\end{theorem}

A generalization of this result, which also replaces convex hulls in Euclidean
space by oriented matroids is due to Holmsen~\cite{Holmsen:2016}.

\begin{theorem}[{\cite[Theorem 1.2]{Holmsen:2016}}]\label{thm:Holmsen}
    Let $M$ be a matroid with rank function $\rho$ and let $\mathcal{O}$ be an
    oriented matroid of rank $d$, both defined on the same ground set $V$ and
    satisfying $\rho(V) = d+1$. If every subset $S \subset V$ such that
    $\rho(V\setminus S) \leq d-1$ contains a positive circuit of
    $\mathcal{O}$, then there exists a positive circuit of $\mathcal{O}$
    contained in an independent set of $M$. 
\end{theorem}

Note that Theorem~\ref{thm:Holmsen} is not a strengthening of
Theorem~\ref{thm:KM}! Figure~\ref{fig:partition-strong-Barany} illustrates
this.
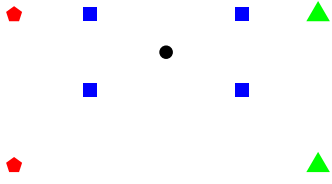
\begin{figure}
    \begin{tikzpicture}
        \tikzset{RedOuter/.style = {regular polygon,regular polygon sides=5,
                inner sep=1.8,fill=red}}
        \tikzset{BlueOuter/.style = {regular polygon,regular polygon sides=4,
                inner sep=1.8,fill=blue}}
        \tikzset{GreenOuter/.style = {regular polygon,regular polygon sides=3,
                inner sep=1.8,fill=green}}
        \tikzset{Origin/.style = {circle, fill=black, inner sep=1.8}}

        \coordinate (origin) at (0,0.5){};

        \coordinate (A1) at (-2,1);
        \coordinate (A2) at (-2,-1);

        \coordinate (B1) at (2,1);
        \coordinate (B2) at (2,-1);

        \coordinate (C1) at (1,1);
        \coordinate (C2) at (1,0);
        \coordinate (C3) at (-1,1);
        \coordinate (C4) at (-1,0);

        \node[Origin] at (origin){};
        \node[RedOuter] at (A1){};
        \node[RedOuter] at (A2){};
        \node[GreenOuter] at (B1){};
        \node[GreenOuter] at (B2){};
        \node[BlueOuter] at (C1){};
        \node[BlueOuter] at (C2){};
        \node[BlueOuter] at (C3){};
        \node[BlueOuter] at (C4){};
    \end{tikzpicture}
    \caption{A configuration fulfilling the prerequisites of
    Theorem~\ref{thm:strong+Barany} but not of
    Theorem~\ref{thm:reformulated+KM}.} \label{fig:partition-strong-Barany}
\end{figure}

Observe that the union of two cocircuits in a partition matroid of rank $d+1$
is a minimal set whose complement has rank $d-1$.  This yields the matroid
generalization of Theorem~\ref{thm:strong+Barany}. 

\begin{corollary} \label{cor:corank2sets}
    Let $M$ be a matroid on the ground set $E$ and $V : E \to \RR^d$.  If the
    origin is contained in the convex hull $\conv(V(S))$ for each set $S$ that
    meets every basis in at least $2$ elements, then there is a basis $B_0$
    such that $0 \in \conv(V(B_0))$.
\end{corollary}

For tropical convexity, we obtain a matroid generalization of the union of two color classes, which is slightly weaker than the condition in Corollary~\ref{cor:corank2sets}. 

\begin{theorem} \label{thm:tropical+weak+Holmsen}
    Let $M = (E,\Ind)$ be a matroid of rank $d+1$ and $V : E \to \TTmax^d$.
    Assume that
    \begin{equation}
        0 \in \tconv(V(C \cup D)) \quad \text{ for any two distinct cocircuits
        } C, D \text{ of } M.
    \end{equation}
    Then there exists a basis $B_0$ such that $0 \in \tconv(V(B_0))$.
\end{theorem}
\begin{proof}
    Again using Lemma~\ref{lem:tropical+Farkas}, we translate the claim into a
    problem on the covector graph $G_V$ of $0$ with respect to $V(E)$.  Let $B$
    be a basis whose neighborhood $\nbhdG(B)$ does not cover $[d+1]$.  We pick
    a surjective map $\phi$ from $B$ to $\nbhdG(B)$.  Since $|B| = d+1 >
    |\nbhdG(B)|$, there are two elements $p,q \in B$ with $\phi(p) = \phi(q)$.
    Let $C_p, C_q$ be cocircuits with $C_p \cap B = \{p\}$ and $C_q \cap B =
    \{q\}$. As $\nbhdG(C_p \cup C_q) = [d+1]$, there is an element $r \in C_p
    \cup C_q$ with $\nbhdG(r) \setminus \nbhdG(B) \neq \emptyset$.  We can
    assume that $r \in C_p$.  We conclude that $B \setminus \{p\} \cup \{r\}$
    is a basis and covers more elements than $B$.
\end{proof}

\subsection{Support complexes} \label{sub:homological}

The proofs in~\cite{KalaiMeshulam:2005} and~\cite{Holmsen:2016} rely on
topological methods. The \Def{support complex} (or
\Def{avoiding complex} in~\cite{AdiprasitoPadrolSanyal:2019}) of $V : E \to
\RR^d$ is the abstract simplicial complex $\Delta_V \subseteq 2^E$ with
simplices
\[
    \sigma \in \Delta_V \quad \Leftrightarrow \quad 0 \not\in \conv(V(\sigma))
\]
for $\sigma \subseteq E$. The structure of support complexes is special. In
both papers, the proofs make use of the fact that the complexes are
$d$-collapsible or \emph{near $d$-Leray}, which is a condition on the vanishing of
the homology of links of $\Delta_V$.

In the tropical setting, the \Def{tropical support complex} can be read off
the covector graph $G_V$
\[
    \Delta_V \ = \ \{ \sigma \subseteq E : \nbhdG(\sigma) \neq [d+1] \}  .
\]

The following result shows that the topological approach does not work in the
tropical setting.

\begin{proposition}\label{prop:support}
    For any simplicial complex $\Delta \subseteq 2^{E}$ with $m$ facets, there
    is $V : E \to \TTmax^{m-1}$ such that $\Delta = \Delta_V$.
\end{proposition}
\begin{proof}
    Let $\sigma_1,\dots,\sigma_m \subseteq E$ be the facets of $\Delta$.
    Consider the bipartite graph $G$ on $E \sqcup [m]$ such that the
    neighborhood of $i \in [m]$ is precisely $E \setminus \sigma_i$. It is
    easy to find $V : E \to \TTmax^{m-1}$ such that $G = G_V$. The
    inclusion-maximal subsets $\sigma \subseteq E$ whose neighborhoods fail to
    cover $[m]$ miss some $i \in [m]$ and hence the facets of $\Delta_V$ are
    exactly $\sigma_1,\dots,\sigma_m$.
\end{proof}

\section{Tropicalized Colorful Linear Programming} \label{sec:LP}
Colorful linear programming was introduced in~\cite{BaranyOnn:1997} as an
algorithmic framework originating from the Colorful Carath{\'e}odory
Theorem~\cite{Barany:1982}: Given $C_1,\dots,C_k \subset \RR^d$ and $b \in
\RR^d$, decide if there is $p_i \in C_i$ for $i=1,\dots,k$ such that $b \in
\conv(p_1,\dots,p_k)$. If $k=d+1$ and $C_1 = C_2 = \cdots = C_{d+1}$, then
this is precisely linear programming.

Many interesting problems can be modelled as colorful linear programs and it
continues to be a topic of active research;
see~\cite{MeunierSarrabezolles:2018}.  In this section, we consider tropical
colorful linear programming and we show that it is NP-complete.

For classical colorful linear programming, NP-completeness was already shown
by B\'ar\'any and Onn~\cite{BaranyOnn:1997}. It might therefore not come as a
surprise that the same holds true in the tropical setting. This intuition,
however, is misguided. To set the stage, we start with a discussion on
tropical integer linear programming that, unlike its classical counterpart, is
in the complexity class $P$.

\subsection{Tropicalized Integer Programming} 
\newcommand\HG{\mathcal{H}}%
The 3-dimensional matching problem is one of the classical NP-complete
problems whose hardness was already established by Karp in his seminal
paper~\cite{Karp:1972}. 
In the exact cover version, we are given three
disjoint sets $A,B,C$ of the same finite cardinality $k$. Given a $3$-uniform
$3$-partite hypergraph $\HG$ on $U := A \cup B \cup C$, namely a collection
of hyperedges $\HG$, where $|h \cap A| = |h \cap B| = |h \cap C| = 1$ for each
$h \in \HG$. A \Def{$\mathbf 3$-dimensional matching} is a subset $M \subseteq
\HG$ such that
\[
    \bigcup_{w \in M} w = A \cup B \cup C \qquad \text{ and } \qquad u \cap v =
\emptyset \quad \text{ for all } u,v \in M, u \neq v.
\]
The algorithmic problem can be formulated as a $0/1$-integer program where not
only the variables but also the coefficients are from $\{0,1\}$.
We introduce a variable $x_h$ for each hyperedge $h \in \HG$. The existence of
a $3$-dimensional matching is equivalent to the feasibility of the system
\begin{equation} \label{eq:3dmIP}
  \begin{aligned}
    \sum_{h \ni w} x_h = 1 & \text{ for all } w \in A \cup B \cup C
    , \\
    x_h \in  \{0,1\}  & \text{ for all } h \in \HG.
  \end{aligned}
\end{equation}
This demonstrates that already this restricted version of integer linear
programming is NP-complete.  

Butkovic~\cite{Butkovic:2019} proposed a version of tropical integer linear
programming based on the lattice points $\mathbb{Z} \subset \mathbb{R}$. 
However, unlike for classical integer programming, these lattice points do not behave well with the tropical operations `$\max$' and `$+$'.

A different approach comes from observing that $-\infty$ and $0$ are the tropical additive and multiplicative elements, respectively. 
The corresponding class of tropical $0/1$-integer programs are
systems of inequalities where the coefficients are restricted to $0$ or
$-\infty$, the neutral elements for tropical addition and multiplication.

For $J \subseteq [n]$ a tropical equation
\[
    \bigoplus_{k \in J} x_k \ = \ 0
\]
with the additional condition $x_k \in \{-\infty,0\}$ corresponds to a
system of inequalities
\begin{equation} \label{eq:zero+inf+inequality}
    x_i \ \le \ \max_{k \in J \setminus i} x_k  \, ,
\end{equation}
with $x_k \in \{-\infty,0\}$, for each $i \in J$.

\newcommand\True{\texttt{true}}%
\newcommand\False{\texttt{false}}%
Associating a Boolean variable $z_i \in \{\bot,\top\}$ for each $x_i$ with
$z_i = \bot$ if and only if $x_i = -\infty$, we can interpret such a tropical
inequality as a Boolean formula
\[
    z_i \Rightarrow \bigvee_{j \in J} z_j
    \quad \text{ or, equivalently, } \quad
    \neg z_i \vee \bigvee_{j \in J} z_j \enspace .
\]

This is a dual Horn clause. A system of
equations~\eqref{eq:zero+inf+inequality} comprises a conjunction of these
clauses. Hence, the feasibility of such a tropical 'integer' program is
equivalent to the satisfiability problem for Horn formulae. This is known to
be linear-time solvable~\cite{DowlingGallier:1984} but also P-complete as
illustrated in the recent survey~\cite{AusielloLaura:2017}. 

\begin{theorem}
    Tropical $\{-\infty,0\}$-integer linear programs can be solved in polynomial
    time.
\end{theorem}
A slightly different point of view is taken in~\cite{MoeSkutStork}. They
consider scheduling with AND-OR networks which in its most general form
is equivalent to tropical linear programming. The problem, which they consider
in~\cite[Theorem~3.1]{MoeSkutStork}, is essentially the same as here.  The
unifying language of hypergraph reachability was already used
in~\cite{AllamigeonGaubertGoubault:2013} to express a similar problem for
tropical polyhedra.  It can also be seen from the point of constraint
satisfaction problems, see~\cite{BodirskyMamino:2018}. 

\begin{example} \label{ex:trivial+tropical+3dm}
    The direct `tropicalization' of~\eqref{eq:3dmIP} results in
    \[
    A \odot x \ = \ \Zero \quad \text{ with }x \in \{-\infty,0\}^{\HG},
    \]
    where $A \in \{-\infty,0\}^{U \times \HG}$ is the hyperedge-node-incidence
    matrix of the hypergraph $\HG$.

    The idempotency of $\max$ significantly weakens constraints.  While the
    equations arising in~\eqref{eq:3dmIP} enforce that exactly one part of the
    tripartition $A \cup B \cup C$ is hit, the tropical equations only encode
    the covering condition. In particular, already the choice $x = \Zero$
    serves as solution. As stated in the discussion
    after~\cite[Corollary~3.1.3]{Butkovic:2010}, this problem becomes
    NP-complete only by imposing an additional minimality condition. 
\end{example}

\subsection{Tropicalized Colorful Linear Programming} \label{sub:tropicalized+CLP}
The tropical version of the `positive dependence version' treated
in~\cite[Sect.~5]{BaranyOnn:1997} is the following problem. Let $A \in
\TTmax^{d \times n}$ and $C_1 \sqcup C_2 \sqcup \cdots \sqcup C_r = [n]$ a
partition. Find $x \in \TTmax^{n}$ such that 
\begin{equation} \label{eq:solution+matching}
    A \odot x \ = \ \Zero  \quad \text{ and } \quad |\supp(x) \cap C_i | = 1 \text{
    for all } i =1,\dots,r \, .
\end{equation}
Here $\supp(x) = \{ i : x_i \neq -\infty \}$. The following is a tropical
analogue to \cite[Theorem 5.1]{BaranyOnn:1997}.

\begin{theorem} \label{prop:one+sided+colourful}
    Tropical colorful linear programming is NP-complete.
\end{theorem}

Note that a similar argument on NP-completeness of \emph{unique solvability}
already occurs in~\cite[Section 3.1]{Butkovic:2010}.

\begin{proof}
    It is clear that the feasibility of some $x \in \TTmax^n$ can be checked
    in polynomial time and the NP-hardness is the nontrivial part.
    To prove hardness, we use a similar construction as in
    Example~\ref{ex:trivial+tropical+3dm}. Let $\HG$ be $3$-uniform
    $3$-partite hypergraph on nodes $U = A \sqcup B \sqcup C$ with $|A| = |B|
    = |C| = k$. For every edge
    $h = \{a,b,c\} \in \HG$, define $v_h \in \TTmax^{U}$ by
    \[
        (v_h)_r \ = \ 
            \begin{cases}
                0 & \text{ if } r \in \{a,b,c\} \\
                -\infty & \text{ otherwise.} 
            \end{cases}
            \]
    For $V = \{ v_h : h \in \HG \}$, we use~\eqref{eq:aff+point+sector} to see
    that the covector graph $G_V$ for the point $p = 0$ is the bipartite graph
    on nodes $\HG \sqcup U$ with edges $(h,a)$, $(h,b)$ and $(h,c)$ for
    every $h = \{a,b,c\} \in \HG$. In particular, every $h \in \HG$ has three
    incident edges in $G_V$. 
    
    Now define color classes $C_a = \{ (a,b,c) \in \HG : b \in B, c \in C\}$
    for $a \in A$ and let $x \in \TTmax^\HG$ such that
    \[
        \bigoplus_{h \in \HG} v_h \odot x_h \ = \ 0
    \]
    with $|\supp(x) \cap C_a| = 1$ for all $a \in A$. Then
    Lemma~\ref{lem:tropical+Farkas} yields that $M = \supp(x)$ covers $A$,
    $B$, and $C$. Since $|M| = |A| = |B| = |C|$, it follows that $M$ is in
    fact a matching.
\end{proof}

\begin{remark}
    A potential application of the above described version of tropical
    colorful linear programming to \emph{constrained scheduling problems}
    arises in the following way.  By replacing
    everything with its negative, we can switch from $\max$ to $\min$; in
    particular, this involves replacing $-\infty$ by $+\infty$.  Consider
    again Equation~\eqref{eq:solution+matching}.  We want to finish $d$ jobs
    at fixed termination times (the vector $\Zero$).  They are based on fixed
    processing times (the entries of $A$).  The columns of $A$ correspond to
    $n$ different ways to finish a job.  For each job (resp. row), at least
    one way has to be finished at the termination time.  The color classes can
    be used to model additional production constraints.  For each way of
    finishing a job (resp. the columns of $A$), only one of the potential
    approaches in a color class $C_i$ for $i \in [r]$ can be chosen.  This
    restriction could be imposed if all approaches in a color class are
    executed on the same unit. 
\end{remark}

\section{Afterthoughts and questions}\label{sec:end}

\subsection{Polymatroid Carath\'eodory theorem}

Theorems~\ref{thm:KM} and~\ref{thm:tropicalKM} can be phrased completely in
terms of the rank function $\rho : 2^E \to \ZZ_{\ge0}$ as independent sets
satisfy $\rho(I) = |I|$. In this form, Theorem~\ref{thm:tropicalKM} is
reminiscent of Rado's generalization of Hall's marriage theorem.
Welsh~\cite{Welsh:1971} further generalized Rado's result to polymatroids. A
\Def{polymatroid} is a function $f : 2^E \to \ZZ_{\ge0}$ that is monotone with
respect to inclusion and submodular $f(A \cup B) + f(A \cap B) \le f(A) +
f(B)$ for all $A,B \subseteq E$. See~\cite[Sect.~11.2]{Oxley:2011} for a
discussion of all three results.

\begin{theorem} \label{thm:independent+transversal}
    Let $(A_j)_{j \in J}$ be a family of non-empty subsets of $E$ and let $f :
    2^E \to \ZZ_{\ge0}$ be a polymatroid. Then there is a choice $e_j \in A_j$
    for $j \in J$ such that
    \[
        f(\{ e_j : j \in K\}) \ \ge \ |K| \quad \text{ for all } K \subseteq J
    \]
    if and only if
    \[
        f(A(K)) \ \ge \ |K| \quad \text{ for all } K \subseteq J\ ,
    \]
    where $A(K) = \bigcup_{j \in K} A_j$.
\end{theorem}
For $f = \rho$, this is Rado's result. For $f(A) = |A|$, this gives Hall's
marriage theorem.

\begin{question}
    Can Theorems~\ref{thm:KM} and~\ref{thm:tropicalKM} can be further
    generalized to polymatroids. 
\end{question}

A first attempt for a very restricted class of polymatroids was done
in~\cite{Sheikh}.
While Theorem~\ref{thm:independent+transversal} is actually a consequence of Edmond's matroid intersection, we also get another connection with the interplay of two matroids arising from the consideration of generic configurations.
As we saw in Section~\ref{sub:greedy}, the containment property for the tropical convex hull translated to a covering property in the associated covector graph.

For the case that $0$ is in generic position with respect to $V$, we can give
a necessary and sufficient condition.  The genericity means that in the
covector graph $G$ each node in $E$ has degree $1$. Equivalently, the set
system $(\nbhdG(i))_{i\in [d+1]}$ forms a partition matroid on $E$. 

The point $p$ is contained in the tropical convex hull of $V(I)$ for an
independent set $I$ of $M$ if and only if $I$ is a spanning set of the
partition matroid $(\nbhdG(i))_{i\in [d+1]}$. 

This means that $(\nbhdG(i))_{i\in [d+1]}$ contains a transversal $(x_i)_{i
\in [d+1]} \subseteq E$ such that $\{x_i \colon i \in [d+1]\}$ is independent
in $M$. Equivalently, there is a transversal $(x_i)_{i \in [d+1]}$ such that
\[
    \rho\left(\{x_j \colon j \in J\}\right) \geq |J| \qquad \forall J
    \subseteq [d] \enspace .
\]
Indeed, the inequality has to be fulfilled with equality.  Now,
Theorem~\ref{thm:independent+transversal} yields the following equivalence.

\begin{proposition}
    If $0$ is in generic position with respect to $V$ then $0$ is in the
    tropical convex hull of an independent set if and only if
    \[
        \rho( \cup_{j \in J} \nbhdG(j) ) \geq |J|  \qquad \forall J \subseteq
        [d] \enspace .
    \]
\end{proposition}

\subsection{Convex geometries}

A \Def{convex geometry} or \Def{anti-matroid} is a pair $(E,\tau)$ where $\tau
: 2^E \to 2^E$ is a hull operator that satisfies the \emph{anti-exchange
axiom}: for $A \subseteq E$ and $x,y \not \in \tau(A)$
\[
    x \in \tau( A \cup y) \ \Longrightarrow \ 
    y \not\in \tau( A \cup x).
\]
For example, if $E \subset \RR^d$ is a finite set, then $\tau(A) := \conv(A)
\cap E$ defines a convex geometry. The definition of tropical convex hulls and
Lemma~\ref{lem:tropical+Farkas} implies the following.

\begin{proposition}\label{prop:trop_anti}
    Let $E \subseteq \TTmax^d$ be a finite set. Then $\tau(A) := \tconv(A)
    \cap E$ defines a convex geometry.
\end{proposition}

This naturally raises the following question:

\begin{question}
    Is there a generalization of Theorems~\ref{thm:KM} and~\ref{thm:Holmsen}
    for convex geometries?
\end{question}

The first question to be answered here is what should take place of the
dimension? A natural choice is the \Def{Carath\'eodory number} of $(E,\tau)$,
that is, the smallest number $c = c(E,\tau)$ such that for any $A \subseteq E$
and $e \in \tau(A)$, there is $A' \subseteq A$ with $|A'| \le c$ and $p \in
\tau(A')$.  For the convex geometries above and $E$ in general position, the
Carath\'eodory number is $d+1$ and suggest the right generalization.
Unfortunately, the following example shows that the Carath\'eodory number is
not a suitable replacement.

\begin{example}\label{ex:non_cara}
    Let $G = (V,E)$ be a connected graph with distinguished node $r \in V$.  A
    subset $F \subseteq E$ is \Def{feasible} if the edge-induced subgraph
    $G[F]$ is connected and contains $r$. A set $C \subseteq E$ is
    \Def{convex} if $E \setminus C$ is feasible. If we define $\tau(A) =
    \bigcap  \{ C: A \subset C, C \text{ convex}\}$, then $(E,\tau)$ is a
    convex geometry; see Example 2.10 in~\cite{KLS}.

    Consider the convex geometry for the graph in Figure~\ref{fig:non_cara}.
    \begin{figure}[htb]
      \centering
      \tikzset{NodeGreedoid/.style = {circle,fill=black,inner sep=1.5}}
\begin{tikzpicture}
  \coordinate (p1) at (0,0);
  \coordinate (p2) at (2,0);
  \coordinate (p3) at (4,0);
  \coordinate (p4) at (6,0);

  \node[NodeGreedoid,label=left:{$r$}] (1) at (p1){};
  \node[NodeGreedoid] (2) at (p2){};
  \node[NodeGreedoid] (3) at (p3){};
  \node[NodeGreedoid] (4) at (p4){};
  
  \draw[bend left] (1) to node[above] {3} (2);
  \draw[bend right] (1) to node[below] {4} (2); 

  \draw[bend left] (2) to node[above] {1} (3);
  \draw[bend right] (2) to node[below] {2} (3);

    \draw[] (3) to node[below] {0} (4);
  
\end{tikzpicture}
\caption{Graph for Example~\ref{ex:non_cara}.}\label{fig:non_cara}
    \end{figure}
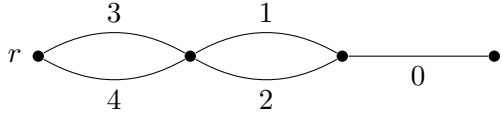
    It can be checked that the Carath\'eodory number for this convex geometry
    is $2$. Indeed, $\{0,1,2\} \subset \tau(\{3,4\})$, $0 \in \tau(\{1,2\})$
    and $3,4$ are the extreme points of $\tau(E)$.  Let $C_1 = \{1,2\}$ and
    $C_2 = \{3,4\}$ be color classes. Then $0 \in \tau(C_1) \cap \tau(C_2)$
    but $0 \not \in \tau(\{p_1, p_2\})$ for any choice $p_i \in C_i$.
\end{example}

The Radon number for this example is $3$. It would be interesting to know if
the Radon number is a suitable replacement for dimension.

Let us finally note that the topological approach of~\cite{KalaiMeshulam:2005}
and~\cite{Holmsen:2016} does not apply. For a convex geometry $(E,\tau)$ and
$p \in E$, one defines the support complex $\Delta_p := \{ A \subseteq E : p
\not\in \tau(A)\}$. However, Propositions~\ref{prop:trop_anti}
and~\ref{prop:support} imply that any simplicial complex is the support
complex of a convex geometry.

\subsection{Examples and Matroid simplicial depth}
It turns out that, as compared to the Colorful Carath\'eodory
Theorem~\ref{thm:Barany}, it is quite difficult to construct non-trivial
instances for Theorem~\ref{thm:KM} where $M$ is not the partition matroid.

Our reformulation of Theorem~\ref{thm:KM} in Theorem~\ref{thm:reformulated+KM} shows
that the conclusion is trivial for $\Cocircuits \cap \Ind \neq \emptyset$. This is
equivalent to the statement that for every basis $B \in \Bases$ there is a
basis $B' \in \Bases$ with $B \cap B' = \emptyset$. For example, if $M$ is a
graphical matroid associated to a graph $G$, then for every node $v \in V(G)$,
the edges incident to $v$ give a cocircuit, which can be completed to a basis.
So graphical matroids do not give non-trivial instances for
Theorem~\ref{thm:KM} independent of the choice of $V : E \to \RR^d$.

\begin{question}
    Find a natural class of instances of Theorem~\ref{thm:KM}.
\end{question}

Deza et al.~\cite{DezaEtal2006} introduced the \Def{colorful simplicial depth}
of a colorful configuration $C_1,\dots,C_{d+1} \subset \RR^{d}$ as the maximal
number of colorful simplices intersecting in a point. They conjectured upper
and lower bounds for color configurations in $\RR^d$ that we shown to be true;
see~\cite{AdiprasitoPadrolSanyal:2019, Sarrabezolles2015}. For any fixed
matroid $M$, one can introduce the analogous notion of \emph{matroid
simplicial depth}.

\begin{question}
    What are lower and upper bounds on the matroid simplicial depth?
\end{question}

The techniques in~\cite{AdiprasitoPadrolSanyal:2019} give an approach: For $M =
(E,\Ind)$ and $V : E \to \RR^d$, define the \Def{M-avoiding complex}
\[
    \Delta_M \ := \ \{ I \in \Ind : 0 \not\in \conv(V(I)) \} \ = \ \Delta_V
    \cap \Ind \, .
\]
This is a subcomplex of the independence complex $\Ind$ and the proof of
Lemma~2.2 in~\cite{AdiprasitoPadrolSanyal:2019} yields

\begin{proposition}
    Let $M$ be a matroid and $V : E \to \RR^d$. Then the matroid simplicial
    depth is bounded from above by
    \[
        (-1)^{|E| - \rho(E)}\mu(M^*) + \tilde\beta_{d-1}(\Delta_M) \, ,
    \]
    where $\mu(M^*)$ is the M\"obius invariant of the dual matroid and $\tilde
    \beta_{d-1}$ is the reduced Betti number (over $\ZZ_2$) of $\Delta_M$.
\end{proposition}
\begin{proof}
    The proof of Lemma~2.2 in~\cite{AdiprasitoPadrolSanyal:2019} verbatimly
    carries over. The missing ingredient is the Euler characteristic of
    $\Ind$. This was determined in~\cite[Prop.~7.4.7]{bjorner} to be 
    $(-1)^{|E| - \rho(E)} \mu(M^*)$.
\end{proof}

To show that this upper bound is tight for the partition matroid it is shown
that any configuration $V : E \to \RR^d$ can be transformed to a natural
configuration that attains the bound and such that in every step the colorful
simplicial depth does not decrease.

\subsection{Signed tropical colorful linear programming}

The notion of tropical convexity considered here is restricted to the tropical non-negative numbers, as $x \geq -\infty$ for all $x \in \TTmax$. Recently in~\cite{LohoVegh:2020}, the extended concept of signed tropical convexity has been developed.
Instead of considering only points with coordinates in $\TTmax$, one glues the non-negative numbers $\TTmax$ with a negative copy $\ominus \TTmax$ at $-\infty$ arriving at $\mathbb{T}_{\pm} = \RR \cup -\infty \cup \ominus\RR$. 

\begin{question}
  Can Theorem~\ref{thm:tropicalKM} and Theorem~\ref{thm:tropical+weak+Holmsen} be generalized to signed tropical convexity? 
\end{question}

It was shown that deciding the containment of the origin $(-\infty,\dots,-\infty)$ in the convex hull of finitely many point in $\mathbb{T}_{\pm}^d$ is equivalent to tropical linear programming. 

We introduced tropical colorful linear programming in Section~\ref{sub:tropicalized+CLP} where no variable occurs on the right hand side in~\eqref{eq:solution+matching}.
Considering the containment for $(-\infty,\dots,-\infty)$ instead of $(0,\dots,0)$ in the sense of signed tropical convexity yields a generalization of the two-sided version of tropical linear programming and our one-sided version of tropical colorful linear programming.

\begin{corollary}
  Signed tropical colorful linear programming is NP-complete. 
\end{corollary}

It was shown in~\cite{LohoVegh:2020} that the most general version of signed tropical linear programming is NP-complete, while a tamer version with non-negative variables is in NP~$\cap$~co-NP as it is equivalent to mean payoff games, see~\cite{MoeSkutStork}. 

The modeling power of non-tropical colorful linear programming was greatly demonstrated in~\cite{MeunierSarrabezolles:2018}. 

\begin{question}
  Which problems can be modeled by signed tropical colorful linear programming? 
\end{question}

\bibliographystyle{amsplain}
\bibliography{TropicalMatroidCaratheodory}

\end{document}